\documentclass[12pt]{amsart}
\usepackage{latexsym}
\usepackage{amsfonts}
\usepackage{amsmath, amssymb, amsthm,amscd,epsfig}
\usepackage{tikz}
\usepackage{mathrsfs}
\usetikzlibrary{arrows,backgrounds,decorations}

\newcommand{\st}{\, \mid \,}

\newcommand{\etale}{$\acute{\textrm{e}}$tale }

\newcommand{\bh}{\mathcal{B}(\mathcal{H})}

\newcommand{\h}{\mathcal{H}}

\newcommand{\dom}{Domain}
\newcommand{\sor}{Source}

\newcommand{\ep}{\varepsilon}

\newcommand{\tr}{\textrm{Tr}}
\newcommand{\gs}{G^s(X,\varphi,Q)}

\newcommand{\D}{\mathfrak{D}}

\tikzstyle{vertex}=[circle,draw,fill=gray!20,thick]
\tikzstyle{goto}=[->,shorten >=1pt,>=stealth,semithick]

\theoremstyle{plain}
\newtheorem{theorem}{Theorem}[section]
\newtheorem{lemma}[theorem]{Lemma}
\newtheorem{proposition}[theorem]{Proposition}

\newtheorem{w-vn}[theorem]{Weyl - von Neumann Theorem}

\theoremstyle{definition}
\newtheorem{definition}[theorem]{Definition}

\theoremstyle{remark}

\newtheorem*{notation}{Notation}

\newtheorem*{Awknowledgements}{Awknowledgements}

\begin{document}

\date{\today}
\title[Spectral triples for hyperbolic dynamical systems]{Spectral triples for hyperbolic dynamical systems}
\author[Michael F. Whittaker]{Michael F. Whittaker}
\address{MICHAEL F. WHITTAKER, School of Mathematics, University of Wollongong, Australia}
\email{mfwhittaker@gmail.com}
\thanks{Research supported in part by ARC grant 228-37-1021, Australia}


\begin{abstract}
Spectral triples are defined for $C^*$-algebras associated with hyperbolic dynamical systems known as Smale spaces. The spectral dimension of one of these spectral triples is shown to recover the topological entropy of the Smale space.
\end{abstract}

\maketitle

\section{Introduction}

A spectral triple $(A,\h,D)$ consists of a faithful representation of a $C^*$-algebra $A$ as bounded operators on a separable Hilbert space $\h$ along with a self-adjoint, unbounded operator on $\h$ satisfying the additional conditions
\begin{enumerate}
\item the set $\{a \in A | [D,a] \in \bh\}$ is norm dense in $A$ and
\item the operator $a(1+D^2)^{-1}$ is a compact operator on $\h$ for all $a$ in $A$.
\end{enumerate}
Alain Connes developed spectral triples as a generalization of a Fredholm module which puts the spectrum of an unbounded, self-adjoint operator at the forefront \cite{Con1}. Using spectral triples Connes was able to recover geometric data from commutative algebras in a framework that extended to the noncommutative case \cite{Con1,Con2}. By now spectral triples are at the forefront in Connes' noncommutative geometry and play a key role in the noncommutative analogue of the calculus. In this paper, we investigate spectral triples for $C^*$-algebras associated with hyperbolic dynamical systems known as Smale spaces. These are the $C^*$-algebras introduced by David Ruelle in his investigation of Gibbs states associated with hyperbolic diffeomorphisms, which are of the type introduced by Connes in connection with foliations \cite{Rue2}.

Let us begin with a heuristic definition of a Smale Space. Suppose $(X,d)$ is a compact metric space and $\varphi:X \rightarrow X$ is a homeomorphism. We say $(X,d,\varphi)$ is a Smale Space if $X$ is locally a hyperbolic product space with respect to $\varphi$; that is, there is a global constant $\ep_X > 0$ such that if $x$ is in $X$ we have two sets $X^s(x,\ep_X)$ and $X^u(x,\ep_X)$ whose intersection is $\{x\}$ and the Cartesian product of these sets is homeomorphic to a neighborhood of $x$. Moreover, for any points $y$ and $z$ in $X^s(x,\ep_X)$ we require that $d(\varphi(y),\varphi(z)) < \lambda^{-1} d(y,z)$ where $\lambda > 1$ is globally defined. Similarly, $X^u(x,\ep_X)$ has the same property if we replace $\varphi$ with $\varphi^{-1}$. We call $X^s(x,\ep_X)$ and $X^u(x,\ep_X)$ the local stable and unstable sets of $x$ respectively.

David Ruelle introduced Smale spaces as a purely topological description of the basic sets of Axiom $A$ diffeomorphisms on a compact manifold \cite{Rue1}. A basic set is a closed, $\varphi$-invariant subset of the manifold but does not need to be a manifold itself. In fact, these sets are usually fractal and have no smooth structure whatsoever. We note that, under mild conditions, Smale spaces are chaotic dynamical systems. We also remark that examples of Smale spaces include shifts of finite type, solenoids, the dynamical systems associated with certain substitution tilings, and hyperbolic toral automorphisms.

Several $C^*$-algebras can be associated with a Smale space. The first algebra we wish to study is the $C^*$-algebra associated with the stable equivalence relation. To do so, it is most convenient to find a transversal, so that it becomes an \'{e}tale equivalence relation. Natural transversals are available as the unstable equivalence classes, but care must be used when defining a suitable topology on the groupoid of stable equivalence restricted to the unstable transversal. The situation is simplified when the transversal is $\varphi$-invariant, so the transversal is defined to be the unstable equivalence classes of a $\varphi$-invariant set of periodic points. A groupoid $C^*$-algebra is produced which first appeared in \cite{PS} and is strongly Morita equivalent to the stable $C^*$-algebra appearing in \cite{Put1}. The unstable $C^*$-algebra is constructed in an analogous fashion. Furthermore, the homeomorphism $\varphi$ gives rise to an automorphism on both the stable and unstable algebras and the crossed products are known as the stable and unstable Ruelle algebras \cite{Put1}. We shall define spectral triples on all of these $C^*$-algebras.

To define a spectral triple we begin by considering specific classes of the stable (unstable) equivalence relation. In our situation, each equivalence class can be associated with a point in the orbit of a periodic point. These periodic orbits are viewed as attractors in the sense that, given $\ep > 0$ and any point $x$ in an equivalence class associated with a periodic point, there is an integer $N$ such that the distance between the periodic orbit and $\varphi^n(x)$ is within $\ep$ for all $n \geq N$. A similar result is true on the unstable equivalence relation provided we replace $\varphi$ with $\varphi^{-1}$. A function is defined on the equivalence classes of an orbit which essentially counts the number of iterations of $\varphi$ required to move each point into a fixed neighbourhood of the orbit of the associated periodic point. Moreover, if the point begins in this fixed neighbourhood then the function will count the number of inverse iterations required to remove the point from the neighbourhood. Using this function, we define a Dirac operator $D$, which gives rise to a spectral triple on the stable algebra. A similar construction defines a spectral triple on the unstable algebra. Furthermore, the Dirac operator $D$ commutes with the automorphism used to define the crossed product Ruelle algebras and therefore the spectral triple defined extends to the Ruelle algebras as well. All of these spectral triples turn out to be $\theta$-summable; that is, the operator $e^{-(1+D^2)}$ is trace class.

A much more desirable property for spectral triples is finite summability. A spectral triple $(A,\h,\D)$ is finitely summable when the operator $(1+\D^2)^{-p/2}$ is trace class for some $p \in \mathbb{R}$. The infimum over all such $p \in \mathbb{R}$ is called the spectral dimension of the spectral triple. Defining a new Dirac operator by $\D = \lambda^D$, where $\lambda >1$ is the local expansive constant of the Smale space and $D$ is the aforementioned Dirac operator, we obtain a finitely summable spectral triple provided we make certain assumptions on the function used to define $D$. We note that this spectral triple does not extend to the Ruelle algebras.

\begin{Awknowledgements}
Great acclamation is due to Ian Putnam who supervised my work during my doctoral studies, from which this note is based.
\end{Awknowledgements}

\section{Smale Spaces} \label{Smale_Spaces}

In the introduction, we gave a heuristic definition of a Smale space and in this section we comment on how to make this definition rigorous, as well as discussing properties required in the sequel. The reader is encouraged to reference \cite{Put1} and \cite{Rue1} for additional details on these remarkable spaces.

To make the definition of a Smale space rigorous requires us to postulate the existence of constants $\ep_X>0$ and $\lambda>1$ as well as a map, called the bracket, satisfying the axioms found in \cite{Put1,Rue1}. The constant $\ep_X>0$ gives specific meaning to the term `local' used in the sequel and $\lambda>1$ is the expansive constant of the Smale space. The idea of the bracket is to encode the local product structure; if $d(x,y) < \ep_X$, then $\{[x,y]\} = X^s(x,\ep_X) \cap X^u(y,\ep_X)$.

The local stable and unstable sets of a point $x$ in $X$ are now defined by
\begin{eqnarray}
\nonumber
& X^s(x,\ep) & =  \{ y \in X | d(x,y) < \ep \textrm{ and } [y,x]=x \} \textrm{ and } \\
\nonumber
& X^u(x,\ep) & = \{ y \in X | d(x,y) < \ep \textrm{ and } [x,y]=x \}
\end{eqnarray}
where $0 \leq \ep \leq \ep_X$. Figure \ref{bracket_fig} illustrates the bracket with respect to these sets.

\begin{figure}[htb]
\begin{center}
\begin{tikzpicture}
\tikzstyle{axes}=[]
\begin{scope}[style=axes]
	\draw[<->] (-3,-1) node[left] {$X^s(x,\ep_X)$} -- (1,-1);
	\draw[<->] (-1,-3) -- (-1,1) node[above] {$X^u(x,\ep_X)$};
	\node at (-1.2,-1.4) {$x$};
	\node at (1.1,-1.4) {$[x,y]$};
	\pgfpathcircle{\pgfpoint{-1cm}{-1cm}} {2pt};
	\pgfpathcircle{\pgfpoint{0.5cm}{-1cm}} {2pt};
	\pgfusepath{fill}
\end{scope}
\begin{scope}[style=axes]
	\draw[<->] (-1.5,0.5) -- (2.5,0.5) node[right] {$X^s(y,\ep_X)$};
	\draw[<->] (0.5,-1.5) -- (0.5,2.5) node[above] {$X^u(y,\ep_X)$};
	\node at (0.7,0.2) {$y$};
	\node at (-1.6,0.2) {$[y,x]$};
	\pgfpathcircle{\pgfpoint{0.5cm}{0.5cm}} {2pt};
	\pgfpathcircle{\pgfpoint{-1cm}{0.5cm}} {2pt};
	\pgfusepath{fill}
\end{scope}
\end{tikzpicture}
\caption{The bracket map}
\label{bracket_fig}
\end{center}
\end{figure}
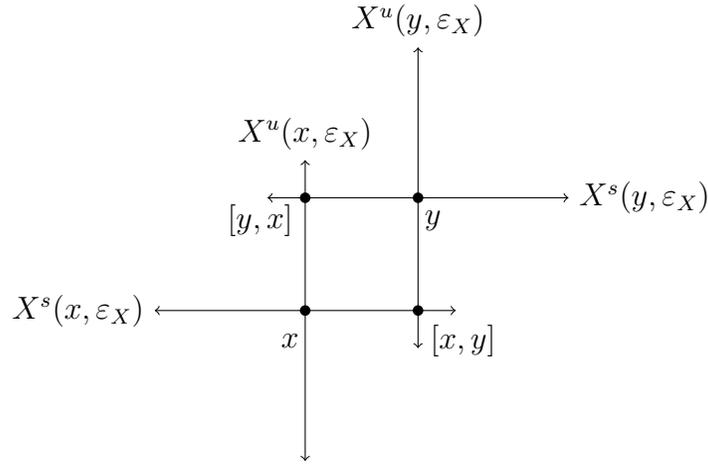

\begin{definition} \label{Smale space}
A dynamical system $(X,d,\varphi)$ having a bracket map is a {\em Smale space}. Moreover, a Smale space is said to be {\em irreducible} if the set of periodic points under $\varphi$ are dense and there is a dense $\varphi$-orbit.
\end{definition}

There are canonical global stable and unstable equivalence relations on $X$. Given a point $x$ in $X$ we define the stable and unstable equivalence classes of $x$ by
\begin{eqnarray}
\nonumber
X^s(x) & = & \{y \in X | \lim_{n \rightarrow +\infty} d(\varphi^n(x),\varphi^n(y)) = 0\}, \\
\nonumber
X^u(x) & = & \{y \in X | \lim_{n \rightarrow +\infty} d(\varphi^{-n}(x),\varphi^{-n}(y)) = 0\}.
\end{eqnarray}
We shall also employ the notation $x \sim_s y$ if $y$ is in $X^s(x)$ and $x \sim_u y$ if $y$ is in $X^u(x)$. To see the connection between the global stable and local stable set of a point, we note that, for any $x$ in $X$ and $\ep >0$, we have $X^s(x,\ep) \subset X^s(x)$. Furthermore, a point $y$ is in $X^s(x)$ if and only if there exists $N \geq 0$ such that $\varphi^n(y)$ is in $X^s(\varphi^n(x),\ep)$ for all $n \geq N$. This nontrivial fact follows from the expansive nature of $\varphi$ in the unstable direction and is most easily observed when $x$ is a fixed point. Indeed, since $y \in X^s(x)$, there exists $N \in \mathbb{N}$ such that $d(\varphi^n(y),x) < \ep$ for all $n \geq N$. Suppose that $\varphi^{n_0}(y)$ is not in the local stable set of $x$ for some $n_0 \geq N$; that is, $[\varphi^{n_0}(y),x] \neq x$. By the definition of the bracket $[\varphi^{n_0}(y),x] \in X^u(x,\ep)$ and it follows that we can find $m$ such that $d(\varphi^m[\varphi^{n_0}(y),x],x) > \ep$ and consequently that $d(\varphi^{m+n_0}(y),x) > \ep$, a contradiction. A slightly more complex argument holds when $x$ is not a fixed point. Making the obvious modifications, the same is true in the unstable situation.

As topological spaces the stable and unstable equivalence classes are quite unseemly with respect to the relative topology of $X$. In fact, if $(X,d,\varphi)$ is irreducible it follows that both the stable and unstable equivalence classes of orbits are dense in $X$ \cite{Rue1}. To rectify this situation we observe that the local stable sets form a neighborhood base for a topology on the global stable sets; that is, given an equivalence class $X^s(x)$, the collection $\{ X^s(y,\delta) | y \in X^s(x) \textrm{ and } \delta >0\}$ is a neighbourhood base for a Hausdorff and locally compact topology on $X^s(x)$. We define a topology on the unstable equivalence classes in an analogous fashion.

\section{$C^*$-algebras of Smale Spaces} 
\label{Algebras}

In this Section we will construct $C^*$-algebras from an irreducible Smale space. These $C^*$-algebras are referred to as the stable and unstable algebras. In \cite{Rue2}, David Ruelle constructed $C^*$-algebras from the stable and unstable equivalence relations. Putnam and Spielberg then refined these constructions in \cite{PS} and defined groupoids that are equivalent, in the sense of Muhly, Renault, and Williams \cite{MRW}, to the stable and unstable groupoids, but which are \'{e}tale. We follow the development in \cite{PS} and the reader is referred there for further properties of these algebras.

The astute reader will have noticed that exchanging the homeomorphism $\varphi$ with $\varphi^{-1}$ interchanges the stable and unstable equivalence relations. This phenomenon persists at the level of the stable and unstable $C^*$-algebras as well. For this reason we omit any discussion of the unstable $C^*$-algebra since we can define the unstable algebra to be the stable algebra of the Smale space with $\varphi$ exchanged with $\varphi^{-1}$.

\subsection{\'{E}tale groupoids on Smale Spaces}

Let $(X,d,\varphi)$ be a Smale space and let $P$ and $Q$ be finite sets of $\varphi$-invariant periodic points. At this point we make no restrictions on the sets $P$ and $Q$, however, in the following section we will add the assumption that $P$ and $Q$ are disjoint. Define
\[ X^s(P) = \bigcup_{p \in P} X^s(p) \quad , \quad X^u(Q) = \bigcup_{q \in Q} X^u(q) \quad , \quad X^h(P,Q) = X^s(P) \cap X^u(Q). \]

\begin{lemma}[\cite{Rue1}] \label{P_Q_dense}
If $(X,d,\varphi)$ is an irreducible Smale space and $P$ and $Q$ are both $\varphi$-invariant sets of periodic points, then $X^h(P,Q)$ is dense in $X$. Moreover, if $P \cap Q = \emptyset$, then $X^h(P,Q)$ does not contain any periodic points.
\end{lemma}

We now define a groupoid on $(X,d,\varphi)$ by
\begin{eqnarray}
\nonumber
& \gs & = \{(v,w) | v \sim_s w \textrm{ and } v,w \in X^u(Q) \}.
\end{eqnarray}
We remark that $\gs$ is a closed transversal to stable equivalence on $(X,d,\varphi)$ in the sense of Muhly, Renault, and Williams \cite{MRW}.

We aim to define an \etale topology $\gs$. Suppose $v \sim_s w$ and $v,w \in X^u(Q)$. Since $v \sim_s w$ it follows that there exists $N$ such that 
\[ \varphi^N(w) \in X^s(\varphi^N(v),\ep_X/2), \]
see Section \ref{Smale_Spaces}. By the continuity of $\varphi$, define $0 < \delta < \ep_X/2$ so that
\[ \varphi^n(X^u(w,\delta)) \subset X^u(\varphi^n(w),\ep_X/2) \qquad \textrm{for all } 0 \leq n \leq N. \]
Given $N,\delta$, we may now define a map $h^s$ on $X^u(w,\delta)$ via
\[ h^s(x) = \varphi^{-N}[\varphi^N(x),\varphi^N(v)]. \]
Let $\delta^\prime = \sup \{d(v,h^s(x)) \mid x \in X^u(w,\delta)\}$. It is shown in \cite{Put1} that the map $h^s:X^u(w,\delta) \to X^u(v,\delta^\prime)$ is a local homeomorphism. An illustration of the map $h^s$ is given in Figure \ref{fig:local_homeo}.
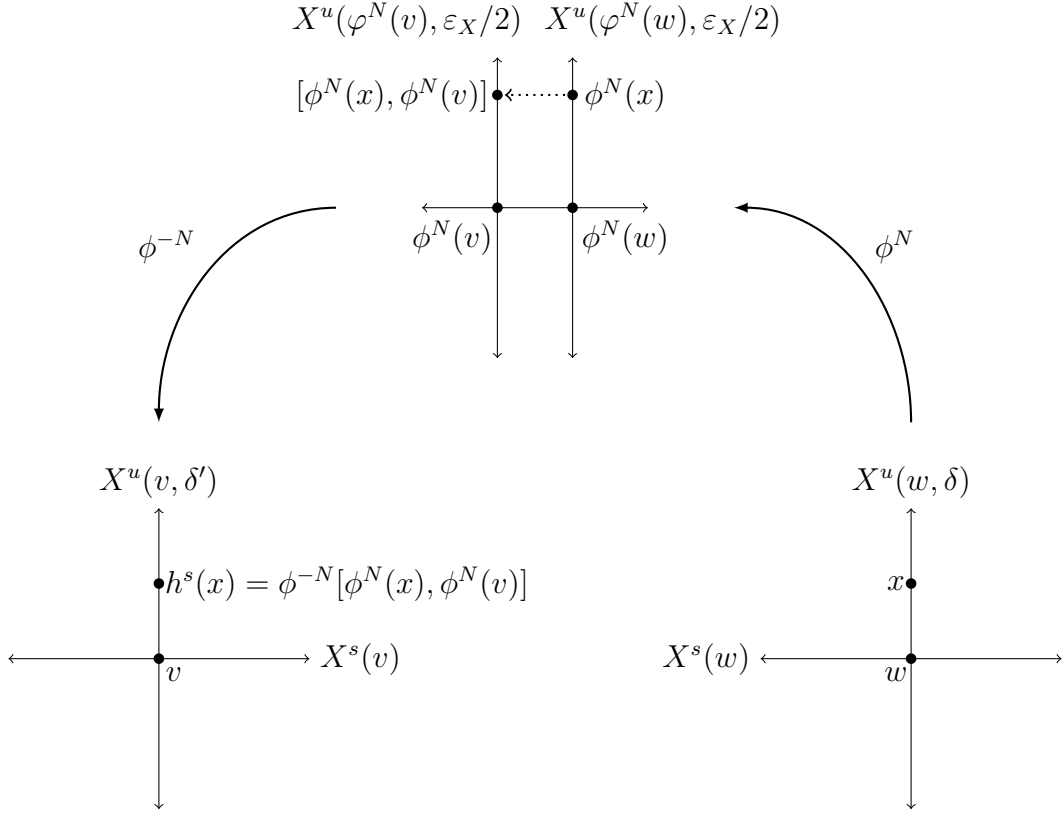
\begin{figure}[htb]
\begin{center}
\begin{tikzpicture}
\tikzstyle{axes}=[]
\begin{scope}[style=axes]
	\draw[<->] (3,0) node[left] {$X^s(w)$} -- (7,0);
	\draw[<->] (5,-2) -- (5,2) node[above] {$X^u(w,\delta)$};
	\node at (4.8,-0.2) {$w$};
	\node at (4.8,1) {$x$};
	\pgfpathcircle{\pgfpoint{5cm}{0cm}} {2pt};
	\pgfpathcircle{\pgfpoint{5cm}{1cm}} {2pt};
	\pgfusepath{fill}
\end{scope}
\begin{scope}[style=axes]
	\draw[<->] (-7,0) -- (-3,0) node[right] {$X^s(v)$};
	\draw[<->] (-5,-2) -- (-5,2) node[above] {$X^u(v,\delta^\prime)$};
	\node at (-4.8,-0.2) {$v$};
	\node at (-2.5,1) {$h^s(x)=\phi^{-N}[\phi^N(x),\phi^N(v)]$};
	\pgfpathcircle{\pgfpoint{-5cm}{0cm}} {2pt};
	\pgfpathcircle{\pgfpoint{-5cm}{1cm}} {2pt};
	\pgfusepath{fill}
\end{scope}
\begin{scope}[style=axes]
	\draw[<->] (-1.5,6) -- (1.5,6) node[right] {$$};
	\draw[<->] (-.5,4) -- (-.5,8) node[above] {$$};
	\draw[<->] (.5,4) -- (.5,8) node[above] {$$};
	\node at (1.7,8.5) {$X^u(\varphi^N(w),\ep_X/2)$};
	\node at (-1.7,8.5) {$X^u(\varphi^N(v),\ep_X/2)$};
	\node at (1.2,5.6) {$\phi^N(w)$};
	\node at (1.2,7.5) {$\phi^N(x)$};
	\node at (-1.1,5.6) {$\phi^N(v)$};
	\node at (-1.9,7.5) {$[\phi^N(x),\phi^N(v)]$};
	\pgfpathcircle{\pgfpoint{-.5cm}{6cm}} {2pt};
	\pgfpathcircle{\pgfpoint{.5cm}{6cm}} {2pt};
	\pgfpathcircle{\pgfpoint{-.5cm}{7.5cm}} {2pt};
	\pgfpathcircle{\pgfpoint{.5cm}{7.5cm}} {2pt};
	\pgfusepath{fill}
	\draw[->,thick,dotted] (0.4,7.5) -- (-0.4,7.5);
\end{scope}
\node (anchor1) at ( 2.5,6) {};
\node (anchor2) at ( 5,3) {}
	edge [->,>=latex,out=90,in=0,thick] node[auto,swap]{$\phi^N$} (anchor1);
\node (anchor3) at ( -2.5,6) {};
\node (anchor4) at ( -5,3) {}
	edge [<-,>=latex,out=90,in=180,thick] node[auto]{$\phi^{-N}$} (anchor3);
\end{tikzpicture}
\caption{The local homeomorphism $h^s: X^u(w,\delta) \rightarrow X^u(v,\delta^\prime)$}
\label{fig:local_homeo}
\end{center}
\end{figure}

\begin{lemma}[\cite{Put1}]\label{local homeo}
Let $v,w$ in $X$ be such that $v \sim_s w$ and $v,w \in X^u(Q)$. There exists $0 < \delta, \delta^\prime \leq \ep_X/2$ and an integer $N$ such that the map $h^s: X^u(w,\delta) \rightarrow X^u(v,\delta^\prime)$ is a local homeomorphism.
\end{lemma}

\begin{theorem}[\cite{Put1}]\label{stable nbhd}
Let $v,w$ in $X$ be such that $v \sim_s w$ and $v,w \in X^u(Q)$ and let $N$, $\delta$, $\delta^\prime$, and $h^s$ be defined by lemma \ref{local homeo}. The collection of sets
\[ V^s(v,w,h^s,\delta) = \{ (h^s(x),x) | x \in X^u(w,\delta), h^s(x) \in X^u(v,\delta^\prime)\} \]
form a neighbourhood base for a topology on $\gs$. In this topology, the range and source maps take each element in the neighbourhood base homeomorphically to an open set in $X^u(Q)$. Moreover, this topology makes $\gs$ a second countable, locally compact, Hausdorff groupoid. That is, $\gs$ is an \etale groupoid. 
\end{theorem}

\subsection{The Stable $C^*$-algebra of a Smale Space}\label{C-algebras}

We aim to study the groupoid $C^*$-algebra of the \etale groupoid $\gs$. To accomplish this, we apply Renault's construction \cite{Ren}.

Let $C_c(\gs)$ denote the continuous functions of compact support on $\gs$, which is a complex linear space. A product and involution are defined on $C_c(\gs)$ as follows, for $f,g \in C_c(\gs)$ and $(x,y) \in \gs$,
\begin{eqnarray}
\nonumber
& f \cdot g(x,y) & = \sum_{(x,z) \in \gs} f(x,z)g(z,y) \\
\nonumber
& f^*(x,y) & = \overline{f(y,x)}.
\end{eqnarray}
This makes $C_c(\gs)$ into a complex $*$-algebra.

We aim to define a norm on $C_c(\gs)$ and then complete $C_c(\gs)$ in this norm to define a $C^*$-algebra. At this point there are several options. First we could look at all possible representations of $C_c(\gs)$ as operators on a Hilbert space. From these Hilbert spaces we obtain a norm and the completion is called the full $C^*$-algebra. Alternatively, we could consider a single representation on each equivalence class, called the regular representation. This gives rise to the reduced norm and the completion is the reduced $C^*$-algebra. In fact, it is shown in \cite{PS} that the groupoid of stable equivalence is amenable so that the full and reduced groupoid $C^*$-algebras are isomorphic.

\begin{definition}
The stable $C^*$-algebra, $S(X,\varphi,Q)$, is the completion of $C_c(\gs)$ in the reduced norm. When no confusion will arise $S(X,\varphi,Q)$ is denoted by $S$.
\end{definition}

A third option is possible when $(X,d,\varphi)$ is irreducible, which is called the fundamental representation \cite{KPW,PS}. We aim to represent $C_c(\gs)$ as operators on the Hilbert space
\[ \h = \ell^2(X^h(P,Q)). \]
To that end, for $f \in C_c(\gs)$ and $\xi \in \h$, define a representation $\pi:C_c(\gs)$ $\rightarrow \bh$ via
\[ \pi(f)\xi(x) = \sum_{(x,y) \in \gs} f(x,y) \xi(y). \]
With this formula, $\pi(f)$ is a bounded linear operator on $\h$. Moreover, we can complete $\pi(C_c(\gs))$ in the operator norm on this Hilbert space to obtain a $C^*$-algebra.

Let us comment on the generality of this construction. In the case that $(X,d,\varphi)$ is mixing, every stable and unstable equivalence class is dense. Moreover, $X^s(P) \cap X^u(Q)$ is dense in $X^u(Q)$ so that $\pi$ is a faithful representation to the reduced $C^*$-algebra and hence is isometric \cite{Ren}. Therefore, the full, reduced, and fundamental $C^*$-algebras of $\gs$ are all isomorphic and $S(X,\varphi,Q)$ is simple. For an irreducible Smale space, it can be shown that there is a canonical decomposition of $X$ into a finite number of distinct mixing components that are cyclically permuted by $\varphi$ so that $X^s(P)$ and $X^u(Q)$ are dense in each component, this remarkable fact is proven in both \cite{Put2} and \cite{Rue1}. Therefore, $\pi$ is faithful and $S(X,d,\varphi)$ is a direct sum of a finite number of simple components. We note that $S(X,\varphi,Q)$ is separable, nuclear, and stable \cite{PS, Put1}.

Each element of $f \in C_c(\gs)$ can be written as a finite sum of functions with support in a neighbourhood base set of the form $V^s(v,w,h^s,\delta)$. We use functions of this form so often in the sequel that we completely describe them in the following lemma, which follows from the definitions.

\begin{lemma}\label{S_basic_function}
Suppose $a$ is a function in $C_c(\gs)$ with support on the basic set $V^s(v,w,h^s,\delta)$ with $v \sim_s w$, $v,w \in X^u(Q)$ and $h^s: X^u(w,\delta) \rightarrow X^u(v,\delta^\prime)$ a local homeomorphism. Then, for $\delta_x \in \h$,
\[ \pi(a)\delta_x = \left\{ \begin{array}{cl}
a(h^s(x),x)\delta_{h^s(x)} & \textrm{ if } x \in X^u(w,\delta) \textrm{ and } h^s(x) \in X^u(v,\delta^\prime) \\
 0 & \textrm{ if } x \notin X^u(w,\delta). \\ \end{array} \right. \]
Define $\sor(a) \subseteq X^u(w,\delta)$ to be the points for which $a$ is non-zero on its domain.
\end{lemma}

We note that every element in $S(X,\varphi,Q)$ can be uniformly approximated by a finite sum of functions supported in a neighbourhood base set. We will usually begin by proving results using these functions and then appealing to continuity for the general result.

\subsection{The Stable Ruelle Algebra of a Smale Space}\label{Ruelle}

A brief construction of the stable Ruelle algebra is given. The Ruelle algebras were first constructed in \cite{Rue2} and alternative constructions were given in \cite{Put1} and \cite{PS} along with many remarkable properties of these $C^*$-algebras. We also note that the stable and unstable Ruelle algebras were shown to satisfy a noncommutative version of Spanier-Whitehead duality in \cite{KPW}.

Given an irreducible Smale space $(X,d,\varphi)$, the homeomorphism $\varphi:X \rightarrow X$ induces an automorphism $\alpha$ on the $C^*$-algebra $S(X,\varphi,Q)$ by
\[ \alpha(a)(x,y) = a(\varphi^{-1}(x),\varphi^{-1}(y)) \]
where $a$ is in $S(X,\varphi,Q)$ and $(x,y)$ are in $\gs$. The homeomorphism $\varphi$ also induces a canonical unitary on the Hilbert space $\h = \ell^2(X^h(P,Q))$ via
\[ u\delta_x = \delta_{\varphi(x)}. \]
Routine calculations show that $(\pi,u)$ are a covariant representation for $(S,\alpha)$.

\begin{definition}[\cite{Put1}]
The stable Ruelle algebra is the crossed product
\[ S(X,\varphi,Q) \rtimes_{\alpha} \mathbb{Z}. \]
Occasionally, we supress the dependence on $Q$ and write $S \rtimes_{\alpha} \mathbb{Z}$.
\end{definition}

\section{Spectral Triples on Smale spaces}

\subsection{Spectral Triples}

Here we define a spectral triple and state some general properties of spectral triples used in the sequel. For a general reference to spectral triples see \cite{Con1}.

To simplify notation we begin to employ $[a,b]$ to denote the commutator $ab-ba$.

\begin{definition}\label{spectral triple}
A {\it spectral triple} $(A,\h,D)$ consists of
\begin{description}
\item[(i)] a separable Hilbert space $\h$,
\item[(ii)] a $*$-algebra $A$ of bounded operators on $\h$,
\item[(iii)] an unbounded self-adjoint operator $D$ on $\h$ such that:
  \begin{description}
  \item[(a)] the set $\{a \in A | [D,a] \in \bh\}$ is norm dense in $A$ and
  \item[(b)] the operator $a(1+D^2)^{-1}$ is a compact operator on $\h$ for all $a$ in $A$.
  \end{description}
\end{description}
\end{definition}

We note that the condition $a(1+D^2)^{-1}$ is a compact operator on $\h$, for all $a$ in $A$, can be replaced with $(1+D^2)^{-1}$ is a compact operator on $\h$, when $A$ is unital.

\begin{definition}\label{summability}
Suppose $(A,\h,D)$ is a spectral triple over a unital $C^*$-algebra $A$ with
\[ \tr((1+D^2)^{-\frac{p}{2}}) < \infty \]
for some positive number $p$. Then the spectral triple is said to be {\it $p$-summable}. Furthermore, the value
\[ \dim_S((A,\h,D)) := \inf\{p>0 | \tr((1+D^2)^{-\frac{p}{2}}) < \infty \} \]
is called the {\it spectral dimension} of the spectral triple. We call $(A,\h,D)$ {\it $\theta$-summable} if, for all $t > 0$,
\[ \tr(e^{-t(1+D^2)}) < \infty. \]
\end{definition}

For spectral triples coming from non-unital $C^*$-algebras the definitions of summability are much more complex. See \cite{Rennie} for details. However, in the case we are interested in, where the $C^*$-algebra is $S(X,\varphi,Q)$, the definition simplifies (since $S(X,\varphi,Q)$ has local units and $X^u(Q)$ is the unit space of the groupoid). For $S(X,\varphi,Q)$, the spectral triple $(S,\h,D)$ is $p$-summable if, for all $a$ in $C_c(X^u(Q))$,
\[ \tr(a(1+D^2)^{-\frac{p}{2}}) < \infty \]
and it is $\theta$-summable if, for all $a$ in $C_c(X^u(Q))$ and for all $t > 0$,
\[ \tr(a e^{-t(1+D^2)}) < \infty. \]

\subsection{Spectral Triples for Smale Spaces} \label{ST_SS}

We wish to construct spectral triples on the stable $C^\ast$-algebras of a Smale space which are geometric and encode the dynamics in a natural way. We begin by constructing a function on $X^s(P)$ and use this function to define a spectral triple on $S(X,\varphi,Q)$.

Let $P$ and $Q$ be finite, mutually distinct, $\varphi$-invariant sets of periodic points. From this point forward, assume that $S(X,\varphi,Q)$ is represented on $\h = \ell^2(X^h(P,Q))$ in order to simplify notation.

Select $0 < \ep \leq \lambda^{-1} \ep_X/2$ where $\lambda > 0$ is the local expansion constant of the Smale space $(X,d,\varphi)$. We aim to define a function $\omega_0:X^s(P) \rightarrow [0,1]$. Consider the closed sets $\overline{X^s(P,\ep)}$ and $X^s(P) \setminus \varphi^{-1}(X^s(P,\ep))$ and observe that these two sets are disjoint. Now an application of Urysohn's lemma implies that there exists a continuous function
\[ \omega_0: X^s(P) \rightarrow [0,1] \]
such that $\omega_0(x) = 0$ for all $x \in \overline{X^s(P,\ep)}$, and $\omega_0(x) = 1$ for all $x \in X^s(P) \setminus \varphi^{-1}(X^s(P,\ep))$. We remark that in practice we may define $\omega_0$ as desired on the complement of our two closed sets, but at this point we merely require that a continuous function exists. A typical function $\omega_0$ is illustrated in Figure \ref{fig:omega_0}, where the notation appearing in the figure is defined as follows.

\begin{notation}
We define the following sets that anticipate the constructions in the sequel and are natural in that context.
\begin{eqnarray*}
\Omega_P  & = & X^s(P,\ep) \setminus P, \\
\Omega_P^c & = & X^s(P) \setminus X^s(P,\ep), \\
E_0 & = & \overline{\varphi^{-1}(X^s(P,\ep)) \cap \Omega_P^c}, \\
E_N & = & \varphi^{-N}(E_0).
\end{eqnarray*}
Let us make some remarks about these sets. Observe that $\omega_0(x) = 0$ for $x \in \Omega_P$ and $\omega_0(x) \geq 0$ for $x \in \Omega_P^c$. In particular, the function $\omega_0$ is defined to either $0$ or $1$ on the points of closure of $E_0$. Also note that
\[ \bigcup_{N \in \mathbb{Z}} E_N = X^s(P) \setminus P. \]
\end{notation}

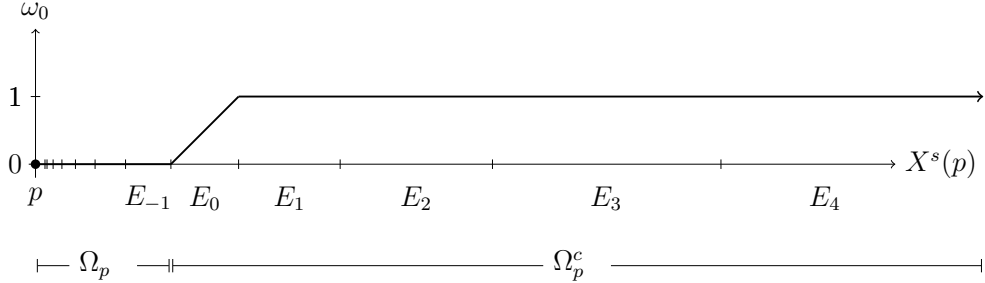
\begin{figure}
\begin{center}
\scalebox{0.90}{
\begin{tikzpicture}
\tikzstyle{axes}=[]
\begin{scope}[style=axes]
	\draw[->] (-0.1,0) -- (12.7,0) node[right] {$X^s(p)$};
	\draw[->] (0,-0.2) node[below] {$p$} -- (0,2) node[above] {$\omega_0$};
	\node at (-0.3,0) {$0$};
	\pgfpathcircle{\pgfpoint{0cm}{0cm}} {2pt};
	\pgfusepath{fill}
	\foreach \y in {1cm}
	  \draw (-2pt,\y) -- (2pt,\y);
	\node at (-0.3,1) {$1$};
	\foreach \x in {0.14cm, 0.17cm, 0.26cm, 0.39cm, 0.59cm, 0.88cm, 1.33cm, 2cm, 3cm, 4.5cm, 6.75cm, 10.125cm}
	  \draw (\x,-2pt) -- (\x,2pt);
	\node at (-0.3,1) {$1$};
	\draw[-,thick] (0,0) -- (2,0);
	\draw[-,thick] (2,0) -- (3,1);
	\draw[->,thick] (3,1) -- (14,1);
	\node at (1.66,-0.5) {$E_{-1}$};
	\node at (2.5,-0.5) {$E_0$};
	\node at (3.75,-0.5) {$E_1$};
	\node at (5.625,-0.5) {$E_2$};
	\node at (8.43,-0.5) {$E_3$};
	\node at (11.66,-0.5) {$E_4$};
	\draw[|-] (0.02,-1.5) -- (0.5,-1.5) node[right] {$\Omega_p$};
	\draw[-|] (1.5,-1.5) -- (1.98,-1.5);
	\draw[|-] (2.02,-1.5) -- (7.5,-1.5) node[right] {$\Omega_p^c$};
	\draw[-|] (8.5,-1.5) -- (13.98,-1.5);
\end{scope}
\end{tikzpicture}}
\end{center}
\caption{The function $\omega_0^p$ for some $p$ in $P$}
\label{fig:omega_0}
\end{figure}

Using $\omega_0$ allows us to encode the dynamics in a natural manner. Let $x$ be a point in $X^s(P) \setminus P$. We aim to define a function which essentially counts the number of iterations it requires for $x$ to be drawn into $\Omega_P$ if it begins in $\Omega_P^c$ and subtracts to number of inverse iterations it requires for $x$ to be removed from $\Omega_P$ if it begins in $\Omega_P$. To that end, define $\omega_s:X^s(P)\setminus P \rightarrow \mathbb{R}$ via
\[ \omega_s(x) = \sum_{n=0}^\infty \omega_0 \circ \varphi^n (x) - \sum_{n=1}^\infty (1-\omega_0) \circ \varphi^{-n} (x). \]
The function $\omega_s$, arising from the function $\omega_0$ in figure \ref{fig:omega_0}, is illustrated in Figure \ref{fig:omega_s}. The following Lemma summarizes the essential properties of $\omega_s$.

\begin{figure}[htb]
\begin{center}
\scalebox{0.90}{
\begin{tikzpicture}
\tikzstyle{axes}=[]
\begin{scope}[style=axes]
	\draw[->] (-0.2,0) node[below] {$p$} -- (12.7,0) node[right] {$X^s(p)$};
	\draw[<->] (0,-3) -- (0,6) node[above] {$\omega_s$};
	\pgfpathcircle{\pgfpoint{0cm}{0cm}} {2pt};
	\pgfusepath{fill}
	\foreach \y in {-2cm,-1cm,1cm,2cm,3cm,4cm,5cm}
	  \draw (-2pt,\y) -- (2pt,\y);
	\node at (-0.3,1) {$1$};
	\node at (-0.3,2) {$2$};
	\node at (-0.3,3) {$3$};
	\node at (-0.3,4) {$4$};
	\node at (-0.3,5) {$5$};
	\foreach \x in {0.14cm, 0.17cm, 0.26cm, 0.39cm, 0.59cm, 0.88cm, 1.33cm, 2cm, 3cm, 4.5cm, 6.75cm, 10.125cm}
	  \draw (\x,-2pt) -- (\x,2pt);
	\node at (-0.3,1) {$1$};
	\draw[<-,thick] (0.59,-3) -- (0.88,-2);
	\draw[-,thick] (0.88,-2) -- (1.33,-1);
	\draw[-,thick] (1.33,-1) -- (2,0);
	\draw[-,thick] (2,0) -- (3,1);
	\draw[-,thick] (3,1) -- (4.5,2);
	\draw[-,thick] (4.5,2) -- (6.75,3);
	\draw[-,thick] (6.75,3) -- (10.125,4);
	\draw[->,thick] (10.125,4) -- (14,4.76);
	\node at (1.66,0.5) {$E_{-1}$};
	\node at (2.5,-0.5) {$E_0$};
	\node at (3.75,-0.5) {$E_1$};
	\node at (5.625,-0.5) {$E_2$};
	\node at (8.43,-0.5) {$E_3$};
	\node at (11.66,-0.5) {$E_4$};
\end{scope}
\end{tikzpicture}}
\end{center}
\caption{The function $\omega_s$ for some $p$ in $P$}
\label{fig:omega_s}
\end{figure}
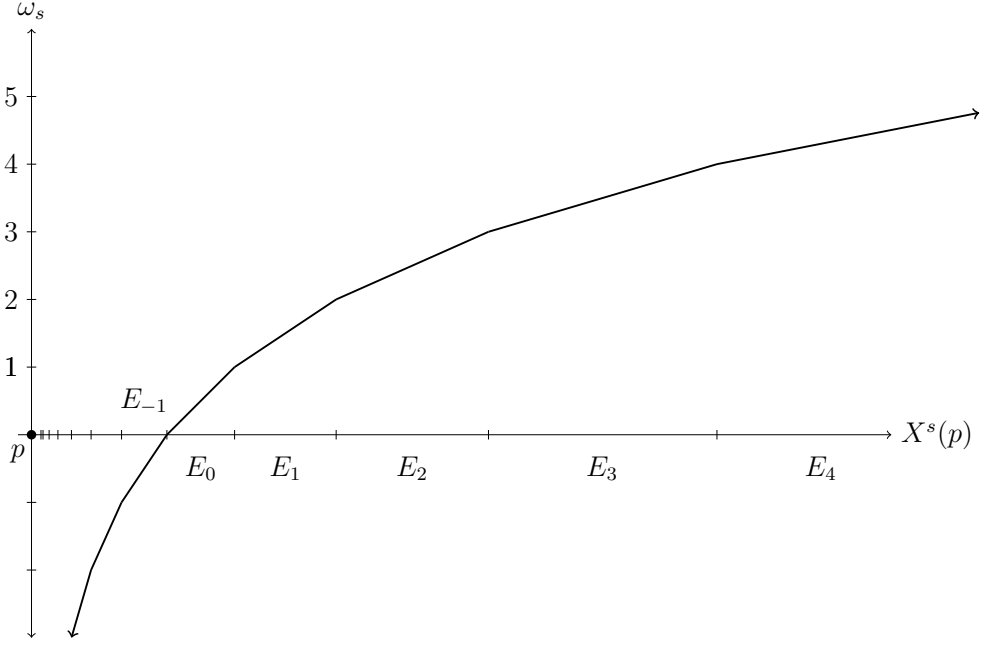

\begin{lemma}\label{omega_s_properties}
Suppose $P$ is a finite, $\varphi$-invariant set of periodic points in a Smale space $(X,d,\varphi)$ and $\omega_s:X^s(P) \setminus P \rightarrow \mathbb{R}$ is defined as above. Then,
\begin{enumerate}
\item $\omega_s(x) \leq 0$ for $x \in \Omega_P$ and $\omega_s(x) \geq 0$ for $x \in \Omega_P^c$,
\item $\omega_s \circ \varphi - \omega_s = 1$,
\item $\omega_s(x) = \omega_0 \circ \varphi^{N}(x) + N$ for $x \in E_N$, and
\item $\omega_s$ is continuous on $X^s(P) \setminus P$.
\end{enumerate}
\end{lemma}

\begin{proof}
First, suppose $x \in E_{-N}$ for some $N \in \mathbb{N}$. Then, the sum on the left, in the definition of $\omega_s$, is zero since $\omega_0(x) = 0$ and $\varphi(x) \in E_{-N-1}$. The sum on the right is finite since $\varphi^N(x) \in E_0$ and $1-\omega_0(\varphi^n(x)) = 0$ for all $n \geq N+1$. Moreover, we have the calculation
\[ \omega_s(x) = - \sum_{n=1}^\infty (1-\omega_0) \circ \varphi^{-n} (x) = -(N-1) - (1-\omega_0)(\varphi^N(x)) = \omega_0(\varphi^{N}(x)) - N. \]
On the other hand, suppose $x \in E_N$ for some $N \in \mathbb{N} \cup {0}$. Then, the sum on the right, in the definition of $\omega_s$, is zero since $1-\omega_0(x) = 0$ and $\varphi^{-1}(x) \in E_{N+1}$. The sum on the left is finite since $\varphi^N(x) \in E_0$ and $\omega_0(\varphi^n(x)) = 0$ for all $n \geq N+1$. Moreover, we have the calculation
\[ \omega_s(x) = \sum_{n=0}^\infty \omega_0 \circ \varphi^n (x) = \omega_0(\varphi^N(x)) + N. \]
This proves that $\omega_s$ is well-defined and the first three statements in the Lemma. For the fourth, we observe that $\omega_s(x)$ is continuous on $E_N$, for all $N \in \mathbb{Z}$, since $\omega_0$ is continuous on $E_0$. If $E_k \cap E_{k+1} = r$, for $k \in \mathbb{Z}$, it follows, from the definition of $\omega_s$, that $\omega_s(r)=k+1$. Since $\bigcup_{N \in \mathbb{Z}} E_N = X^s(P) \setminus P$ and $\omega_s$ gives equal value to the common boundary of $E_k$ and $E_{k+1}$ for all $k \in \mathbb{N}$, $\omega_s$ is continuous. 
\end{proof}

We now consider how the function $\omega_s$ interacts with functions in $C_c(\gs)$ supported on basic sets, see Lemma \ref{S_basic_function} for details on basic sets and for a definition of $\sor(a)$ for $a$ in $S(X,\varphi,Q)$.

\begin{lemma}\label{omega_s_interaction}
Suppose $P$ is a finite, $\varphi$-invariant set of periodic points in a Smale space $(X,d,\varphi)$ and $\omega_s:X^s(P) \rightarrow \mathbb{R}$ is defined as above. Let $a \in C_c(\gs)$ be supported on a basic set $V^s(v,w,h^s,\delta)$ so that $\sor(a) \subseteq X^u(w,\delta)$ and $h^s(\sor(a)) \subseteq X^u(v,\delta^\prime)$. Then,
\begin{enumerate}
\item there exists $N \in \mathbb{Z}$ such that $E_n \cap \sor(a) = \emptyset$ for all $n \leq N$,
\item for all $N$, the number of points in $E_N \cap \sor(a)$ is finite,
\item if $x \in E_N \cap \sor(a)$, then there exists $K \in \mathbb{N}$ such that $h^s(x) \in \bigcup_{k=-K}^K E_{N+k}$.
\end{enumerate}
\end{lemma}

\begin{proof}
Since $\sor(a)$ is a pre-compact subset of $X^u(Q)$ and $P \nsubseteq X^u(Q)$, define $\delta_0>0$ via $\delta_0 = \inf \{d(p,x) | p \in P, x \in \sor(a)\}$. Now there exists $N \in \mathbb{Z}$ such that $E_n \subset X^s(P,\delta_0)$ for all $n \leq N$. Therefore, $E_n \cap \sor(a) = \emptyset$ for all $n \leq N$ as well. For the second claim, $\sor(a)$ and $E_N$ are transverse and compact for all $N$, it follows that $E_N \cap \sor(a)$ is finite. 

For the third claim, first note that given that $a$ is supported in $V^s(v,w,h^s,\delta)$, there exists $M$ such that for all $x \in \sor(a)$, $d(\varphi^M(h^s(x)),\varphi^M(x)) < \ep_X/2$ and $[\varphi^M(x),\varphi^M(h^s(x))]=\varphi^M(x)$. Moreover, it follows that there exists $L \in \mathbb{N}$ such that $D > \ep_X$ where $D = \sup\{d(y,y^\prime) \st y,y^\prime \in E_L \textrm{ and } [y,y^\prime]=y\}$; that is, we can find $E_L$ so that $E_L$ has diameter larger than $\ep_X$ on the stable set of each periodic point $p$ in $P$. Now we claim that if $(h^s(x),x) \in V^s(v,w,h^s,\delta)$ has the property that $x \in E_m$ where $m \geq M+L+1$, then $h^s(x) \in \bigcup_{k=-1}^{1} E_{m+k}$. Indeed, for all $(h^s(x),x)$ we have $d(\varphi^M(h^s(x)),\varphi^M(x)) < \ep_X/2$ and if $\varphi^M(x)$ is in $E_{L+1}$, then by the triangle inequality we have $\varphi^M(h^s(x)) \in \cup_{k=-1}^{1} E_{L+1+k}$. Now applying $\varphi^{-M}$ to $\varphi^M(x)$ and $\varphi^M(h^s(x))$ proves the claim. We are left with the case that $\sor(a) \in E_n$ for $n < M+L+1$. However, combining part ($1$) and ($2$) implies that there are only a finite number of such elements, so that we may define
\[ K = \max\{ 1,|i-j| \st h^s(x) \in E_i, x \in E_j, \textrm{ and } i,j < M+L+1 \}, \]
which is finite. Now $K$ has the property that if $x \in E_N \cap \sor(a)$, then $h^s(x) \in \bigcup_{k=-K}^K E_{N+k}$.
\end{proof}

\subsection{A $\theta$-Summable Spectral Triple} \label{theta_summable_triple}

In this section, we define a spectral triple on $S(X,\varphi,Q)$. The idea is to use $\omega_s$ to define a Dirac operator on $\h = \ell^2(X^h(P,Q))$. Let
\[ D \delta_x = \omega_s(x) \delta_x. \]
The domain of $D$ is given by
\[ \dom(D) = \{ \xi \st \sum_{x \in X^h(P,Q)} \omega_s^2(x)|\xi(x)|^2 < \infty \}. \]
and routine calculations show that $D$ is self-adjoint and unbounded.

\begin{lemma}
For $a \in C_c(\gs)$, the commutator $[D,a]$ is a bounded operator on $\h$.
\end{lemma}

\begin{proof}
Let $a$ in $C_c(\gs)$ be supported on a basic set $V^s(v,w,h^s,\delta)$. By part ($3$) of Lemma \ref{omega_s_interaction}, if $x \in E_N \cap \sor(a)$, then there exists $K \in \mathbb{N}$ such that $h^s(x) \in \bigcup_{k=-K}^K E_{N+k}$. Therefore, for any $x \in E_N \cap \sor(a)$, using part ($3$) of lemma \ref{omega_s_properties}, we compute
\begin{eqnarray*}
\|[D,a] \delta_x\| & = & \|(\omega_s(h^s(x)) - \omega_s(x)) a(h^s(x),x) \delta_{h^s(x)}\| \\
& = & |\omega_s(h^s(x)) - \omega_s(x)| |a(h^s(x),x)| \\
& \leq & |\omega_0(\varphi^{N+K}(h^s(x))) + N+K - (\omega_0(\varphi^{N}(x)) + N)| |a(h^s(x),x)| \\
& \leq & (K+1) |a(h^s(x),x)|.
\end{eqnarray*}
Since $a$ is compactly supported, $|a(h^s(x),x)|$ attains a maximum value. Moreover, everything above is independent of $N$ so that $[D,a]$ is bounded. For the general case we recall that any element of $C_c(\gs)$ is in the span of functions supported on basic sets.
\end{proof}

\begin{proposition}\label{compact operator}
For every $a$ in $S(X,\varphi,Q)$, the operator $a(1+D^2)^{-1}$ is compact on $\h$.
\end{proposition}

\begin{proof}
Let $a_0$ in $S(X,\varphi,Q)$ be supported on a basic set of the form $V^s(v,w,h^s,\delta)$. By part ($1$) of Lemma \ref{omega_s_interaction}, there exists $M$ such that $E_m \cap \sor(a_0) = \emptyset$ for all $m \leq M$. Furthermore, by part ($2$) of Lemma \ref{omega_s_interaction}, the number of elements in $E_N \cap \sor(a_0)$ is finite for all $N$. Now using part ($3$) of Lemma \ref{omega_s_properties}, for $x \in E_N \cap \sor(a_0)$ we have
\[ \| a_0(1+D^2)^{-1} \delta_x \| = \| \frac{a_0(h^s(x),x)}{1+\omega_s^2(x)} \delta_h^s(x) \| \leq \frac{|a_0(h^s(x),x)|}{1+N^2}. \]
Since $a_0$ has compact support, let $A = \sup\{ |a_0(h^s(x),x)| \st x \in \sor(a_0)\}$. Moreover since $h^s$ takes basis vectors to basis vectors and is a homeomorphism from $\sor(a_0)$ to $h^s(\sor(a_0))$ it follows that, restricted to $E_N$,
\[ \| a_0(1+D^2)^{-1} \| \leq \frac{A}{1+N^2}. \]
Therefore, $a_0(1+D^2)^{-1}$ is a norm limit of finite rank operators. Moreover, $a$ in $S(X,\varphi,Q)$ is a norm limit of finite sums of operators of the form $a_0$, so $a(1+D^2)^{-1}$ is compact as well.
\end{proof}

Before arriving at our main theorem for the section, we must delve into a technical result. For an irreducible Smale space $(X,d,\varphi)$, the topological entropy of $(X,d,\varphi)$ is denoted $h(X,\varphi)$ and is the growth rate of the number of essentially different orbit segments of length $N$, for further details see \cite{BS}. We state the following result which is obtained by combining Lemma $5.9$ and Proposition $5.12$ in \cite{Kil}. There are also several similar results in \cite{Men}.

\begin{theorem}[\cite{Kil}]\label{entropy_approx}
Suppose $(X,d,\varphi)$ is an irreducible Smale space with $P$ and $Q$ distinct, finite, $\varphi$-invariant sets of periodic points. Then for any $\delta_0, \delta_1 >0$ and any $w \in X^u(Q)$, $\#\{x \st \varphi^{-N}(X^s(P,\delta_0)) \cap X^u(w,\delta_1)\}$ is finite. Moreover,
\[ \lim_{N \rightarrow \infty} |\frac{1}{N} \log (\#\{x \st \varphi^{-N}(X^s(P,\delta_0)) \cap X^u(w,\delta_1)\}) - h(X,\varphi)| = 0. \]
\end{theorem}

\begin{theorem} \label{theta_triple}
Suppose $(X,d, \varphi)$ is an irreducible Smale space, then $(S, \h, D)$ is a non-unital, $\theta$-summable spectral triple.
\end{theorem}

\begin{proof}
We have shown that $(S,\h,D)$ is a spectral triple. It remains to show that $(S,\h,D)$ is $\theta$-summable. We must show that, for $a$ in $C_c(X^u(Q))$ a positive operator, we have $\tr(ae^{-t(1+D^2)}) < \infty$ for all $t >0$. By part ($1$) of Lemma \ref{omega_s_interaction}, there exists $M$ such that $E_m \cap \sor(a) = \emptyset$ for all $m \leq M$. Furthermore, by part ($2$) of Lemma \ref{omega_s_interaction}, the number of elements in $E_N \cap \sor(a)$ is finite for all $N$. Now using part ($3$) of Lemma \ref{omega_s_properties}, for $x \in E_N \cap \sor(a)$ we have
\begin{eqnarray}
\nonumber
\tr(ae^{-t(1+D^2)}) & = & \sum_{x \in X^h(P,Q)} \left< ae^{-t(1+D^2)} \delta_x,\delta_x \right> \\
\nonumber
& = & \sum_{x \in \sor(a)} \frac{a(x,x)}{e^{t(1+\omega_s^2(x))}} \\
\label{41}
& \leq & \sum_{n=M}^\infty \left( \sum_{\#\{x \st x \in E_n \cap \sor(a)\}} \frac{a(x,x)}{e^{t(1+(n)^2)}}\right).
\end{eqnarray}
Now from Theorem \ref{entropy_approx}, for $\ep > 0$, there exists $N$ such that for all $n \geq N$,
\[ \#\{x \st x \in E_n \cap \sor(a)\} < e^{n(h(X,\varphi) + \ep)}. \]
Therefore, letting $A= \sup\{a(x,x) | x \in \sor(a)\}$, we have
\begin{eqnarray*}
\sum_{\#\{x \st x \in E_n \cap \sor(a)\}} \frac{a(x,x)}{e^{t(1+(n)^2)}} & < & \frac{Ae^{n(h(X,\varphi) + \ep)}}{e^{t(1+(n)^2)}} \\
& = & Ae^{n(h(X,\varphi) + \ep)-t(1+(n)^2)} \\
& < & Ae^{n(h(X,\varphi) + \ep)-tn^2} \\
& = & Ae^{n(h(X,\varphi) + \ep - tn)}.
\end{eqnarray*}
Putting this into (\ref{41}) and letting $R$ denote the first $N-1$ terms of the sum yields
\begin{eqnarray*}
\tr(ae^{-t(1+D^2)}) & = & R + \sum_{n=M}^\infty Ae^{n(h(X,\varphi) + \ep - tn)},
\end{eqnarray*}
which converges since we can choose $N$ sufficiently large that $tN > h(X,\varphi) + \ep$.
\end{proof}

Recall that the stable Ruelle algebra is the crossed product $S \rtimes_{\alpha} \mathbb{Z}$, see Section \ref{Ruelle}.  As operators, $S \rtimes_{\alpha} \mathbb{Z}$ is the completion of $span\{a \cdot u^k \st a \in S(X,\varphi,Q) \textrm{ and } k \in \mathbb{Z} \}$ in the Hilbert space $\h = \ell^2(X^h(P,Q))$ where $u$ is the canonical unitary on $\h$ defined by $u \delta_x = \delta_{\varphi(x)}$. Using $(2)$ in Lemma \ref{omega_s_properties}, we have $[u,D]=u$ so that $\|[u,D]\|=1$. Therefore, $[a\cdot u^k,D] = a[u^k,D]+[a,D]u^k$ is a bounded operator and we obtain a spectral triple on the stable Ruelle algebra as well.

\begin{theorem} \label{Ruelle_triple}
Suppose $(X,d, \varphi)$ is an irreducible Smale space, then $(S \rtimes_{\alpha} \mathbb{Z}, \h, D)$ is a non-unital, $\theta$-summable spectral triple.
\end{theorem}

\subsection{A $p$-Summable Spectral Triple}\label{summable}

In this section we add the hypothesis that the function $\omega_0$ is locally Lipschitz continuous in order to define a summable spectral triple on $S(X,\varphi,Q)$. The added assumption that $\omega_0$ is locally Lipschitz continuous will not restrict the Smale spaces we consider in any way since a locally Lipschitz continuous function can be defined using the Smale space metric.

Let us define $\omega_0$ to be locally Lipschitz continuous; that is, there exists a constant $C_0$ such that if $x,y \in E_0$, $[x,y]=x$, and $d(x,y) < \ep_X/2$, then
\[ |\omega_0(x) - \omega_0(y)| < C_0 d(x,y) \]
where the metric comes from the Smale space itself. In fact, since $(X,d)$ is a compact metric space we can always define such a function using the metric and regarding $E_0$ as a disjoint union of closed sets, one for each element of $P$. Let us also define a constant $C_s = 2KC_0$ where 
\[ K = \max\{k > 0 \st [x,y]=x, d(x,y) < \ep_X/2, \textrm{ with } x \in E_0 \textrm{ and } y \in E_k\}. \]

\begin{lemma}
The function $\omega_s$ is locally Lipschitz continuous on $\cup_{n=0}^\infty E_n$; that is, if $x,y \in \cup_{n=0}^\infty E_n$, $d(x,y) < \ep_X/2$ and $[x,y]=x$, then
\[ |\omega_s(x) - \omega_s(y)| < C_s d(x,y). \]
\end{lemma}

\begin{proof}
First observe that $\omega_s(x)$ is locally Lipschitz continuous, with Lipschitz constant $C_0$, on $E_N$, for all $N \in \mathbb{N}$. Indeed, suppose $x,y \in E_N$ such that $[x,y]=x$. Then, using part ($3$) of Lemma \ref{omega_s_properties},
\begin{eqnarray*}
|\omega_s(x) - \omega_s(y)| & = & |(\omega_0(\varphi^N(x)) + N) - (\omega_0(\varphi^N(y)) + N)| \\
& = & |\omega_0(\varphi^N(x)) - \omega_0(\varphi^N(y))| \\
& < & C_0 d(\varphi^N(x),\varphi^N(y)) < C_0 \lambda^{-N} d(x,y).
\end{eqnarray*}
Now suppose $x,y \in \cup_{n=0}^\infty E_n$, $d(x,y) < \ep_X/2$ and $[x,y]=x$. Then we note that if $x \in E_N$ then $y \in \cup_{k=-K+N}^{K+N} E_k$ where $K$ comes from the definition of $C_s$. The triangle inequality gives the desired result.
\end{proof}

Define an operator $\D$ on $\h = \ell^2(P,Q)$ via
\[ \D \delta_x = \lambda^{\omega_s(x)} \delta_x \]
where $\lambda > 1$ the local growth rate of $(X,d,\varphi)$. The operator $\D$ is also self-adjoint, unbounded, and has dense domain.

\begin{lemma}\label{commutator bounded SDs}
For $a$ in $C_c(\gs)$, the commutator $[a,\D]$ is a bounded operator on $\h$.
\end{lemma}

\begin{proof}
Let $a$ in $C_c(\gs)$ be supported on a basic set $V^s(v,w,h^s,\delta)$, which implies that there exists $M$ such that, for all $(h^s(x),x) \in V^s(v,w,h^s,\delta)$, we have $d(\varphi^M(h^s(x)),\varphi^M(x)) < \ep_X/2$. By Lemma \ref{omega_s_interaction}, for all but a finite number of $(h^s(x),x) \in V^s(v,w,h^s,\delta)$ we have both $x$ and $h^s(x)$ in $\displaystyle{\cup_{m=M}^\infty E_m}$. Suppose we are given such an $(h^s(x),x)$. Without loss of generality suppose $x \in E_k$ and $h^s(x) \in E_j$ where $M \leq k \leq j$. We compute
\begin{eqnarray*}
\|[a,\D] \delta_x \| & = & \| (\lambda^{\omega_s(x)} - \lambda^{\omega_s(h^s(x))}) a(h^s(x),x) \delta_{h^s(x)} \| \\
& = & | \lambda^{\omega_s(x)} - \lambda^{\omega_s(h^s(x))} | | a(h^s(x),x) | \\
& = & \lambda^{\omega_s(x)} |1-\lambda^{\omega_s(h^s(x)) - \omega_s(x)}| | a(h^s(x),x) | \\
& = & \lambda^{k+\omega_0(\varphi^k(x))} |1-\lambda^{\omega_s(M+\varphi^{M}(h^s(x))) - (M + \omega_s(\varphi^{M}(x)))}| | a(h^s(x),x) | \\
& \leq & \lambda^{k+1} |1-\lambda^{\omega_s(\varphi^{M}(h^s(x))) - \omega_s(\varphi^{M}(x))}| | a(h^s(x),x) | \\
& < & \lambda^{k+1} |1-\lambda^{C_s d(\varphi^{M}(h^s(x)),\varphi^{M}(x))}| | a(h^s(x),x) | \\
& < & \lambda^{k+1} |1-\lambda^{C_s \ep_X/2}| | a(h^s(x),x) | \\
& < & \lambda^{k+1} |1-\lambda^{\log_\lambda(1+C_s \ep_X/2)}| | a(h^s(x),x) | \\
& = & \lambda^{k+1} |1-(1+C_s \ep_X/2)| | a(h^s(x),x) | \\
& = & \lambda^{k+1} C_s\ep_X/2 | a(h^s(x),x) | \\
& \leq & \lambda^{k+1} \lambda^{M-k}C_s\ep_X/2 | a(h^s(x),x) | \\
& = & C_s \lambda^{M+1} \ep_X/2 | a(h^s(x),x) | 
\end{eqnarray*}
where $M$ depend only on the set $V^s(v,w,h^s,\delta)$. Since $a$ is compactly supported it attains its maximum. Thus, in this case $[a,\D]$ is bounded.

On the other hand, if $(h^s(x),x)$ is in the finite set where both $x$ and $h^s(x)$ are not in $\displaystyle{\cup_{m=M}^\infty E_m}$, then we can take the maximum value of $ | \lambda^{\omega_s(x)} - \lambda^{\omega_s(h^s(x))} |$ which is bounded simply because it is a finite set. Therefore, $[a,\D]$ is bounded and the Lemma is proven.
\end{proof}

To complete the proof that $(S,\h,\D)$ is a non-unital spectral triple we need only show that $a(1+\D^2)^{-1}$ is a compact operator for every $a$ in $S(X,\varphi,Q)$. The same argument as presented in Section \ref{theta_summable_triple} gives the result. We will now show that $(S, \h, \D)$ is a finitely summable spectral triple. Indeed, $(S, \h, \D)$ is $\log_\lambda(e)h(X,\varphi)$-summable, where $h(X,\varphi)$ is the topological entropy of the Smale space $(X,d,\varphi)$. We note that the factor $\log_\lambda(e)$ is merely a base change from a base $e$ logarithm to a base $\lambda$ logarithm.

\begin{theorem} \label{summable_triple}
Suppose $(X,d, \varphi)$ is an irreducible Smale space, then $(S, \h, \D)$ is a non-unital, $\log_\lambda(e)h(X,\varphi)$-summable spectral triple, where $h(X,\varphi)$ is the topological entropy of the Smale space $(X,d,\varphi)$.
\end{theorem}

\begin{proof}
We have shown that $(S,\h,\D)$ is a spectral triple. It remains to show that $(S,\h,\D)$ is summable. We must show that, for $a$ in $C_c(X^u(Q))$ a positive operator, we have 
\[ \tr(a(1+\D^2)^{-\frac{s}{2}}) < \infty \]
for some $s > 0$. By part ($1$) of Lemma \ref{omega_s_interaction}, there exists $M$ such that $E_m \cap \sor(a) = \emptyset$ for all $m \leq M$. Furthermore, by part ($2$) of Lemma \ref{omega_s_interaction}, the number of elements in $E_N \cap \sor(a)$ is finite for all $N$. Now using part ($3$) of Lemma \ref{omega_s_properties}, for $x \in E_N \cap \sor(a)$ we have
\begin{eqnarray}
\nonumber
\tr(a(1+\D^2)^{-\frac{s}{2}}) & = & \sum_{x \in X^h(P,Q)} \left< a(1+\D^2)^{-\frac{s}{2}}) \delta_x,\delta_x \right> \\
\nonumber
& = & \sum_{x \in \sor(a)} \frac{a(x,x)}{(1+\lambda^{2\omega_s(x)})^{s/2}} \\
\label{52}
& \leq & \sum_{n=M}^\infty \left( \sum_{\#\{x \st x \in E_n \cap \sor(a)\}} \frac{a(x,x)}{(1+\lambda^{2n})^{s/2}} \right).
\end{eqnarray}
Now from Theorem \ref{entropy_approx}, for $\ep > 0$, there exists $N$ such that for all $n \geq N$,
\begin{eqnarray}
\label{53}
e^{n(h(X,\varphi) - \ep)} < \#\{x \st x \in E_n \cap \sor(a)\} < e^{n(h(X,\varphi) + \ep)}.
\end{eqnarray}
Therefore, letting $A = \sup\{a(x,x) \st x \in \sor(a)\}$, we have
\begin{eqnarray*}
\sum_{\#\{x \st x \in E_n \cap \sor(a)\}} \frac{a(x,x)}{(1+\lambda^{2n})^{s/2}} & < & \frac{Ae^{n(h(X,\varphi) + \ep)}}{(1+\lambda^{2n})^{s/2}} \\
& < & \frac{Ae^{n(h(X,\varphi) + \ep)}}{(\lambda^{2n})^{s/2}} \\
& = & \frac{A\lambda^{\log_\lambda(e)n(h(X,\varphi) + \ep)}}{(\lambda^{sn})} \\
& = & A\lambda^{n(\log_\lambda(e)h(X,\varphi) + \log_\lambda(e)\ep)-sn} \\
& = & A\lambda^{n((\log_\lambda(e)h(X,\varphi) + \log_\lambda(e)\ep)-s)}
\end{eqnarray*}
Putting this into (\ref{52}) and letting $R$ denote the first $N-1$ terms of the sum yields
\begin{eqnarray*}
\tr(a(1+\D^2)^{-\frac{s}{2}}) & < & R + \sum_{n=M}^\infty A\lambda^{n((\log_\lambda(e)h(X,\varphi) + \log_\lambda(e)\ep)-s)},
\end{eqnarray*}
which converges geometrically for $s > \log_\lambda(e)h(X,\varphi) + \log_\lambda(e)\ep$. Since this holds for any $\ep > 0$ we have that
\[ \log_\lambda(e)h(X,\varphi) \geq \inf \{s \st \tr(a(1+\D^2)^{-\frac{s}{2}}) < \infty\}. \]
On the other hand, using the other inequality in (\ref{53}), a similar computation shows that
\begin{eqnarray*}
\tr(a(1+\D^2)^{-\frac{s}{2}}) > R + \sum_{n=M}^\infty \frac{\min\{a(x,x)\}\lambda^s}{2} \lambda^{n((\log_\lambda(e)h(X,\varphi) - \log_\lambda(e)\ep)-s)},
\end{eqnarray*}
which converges geometrically only if $s > \log_\lambda(e)h(X,\varphi) - \log_\lambda(e)\ep$ for every $\ep >0$. Therefore, we have
\[ \log_\lambda(e)h(X,\varphi) = \inf \{s \st \tr(a(1+\D^2)^{-\frac{s}{2}}) < \infty\}. \]
\end{proof}

To conclude, we note that it is not obvious that the operator $\D$ gives rise to a spectral triple on the stable Ruelle algebra $S \rtimes_{\alpha} \mathbb{Z}$. We would be very interested to know if it does.

\end{document}